\documentclass[10pt,reqno]{amsart}

\usepackage[a4paper,left=35mm,right=35mm,top=25mm,bottom=25mm,marginpar=25mm]{geometry}

\usepackage{amsfonts}
\usepackage{amsthm}
\usepackage{amssymb}
\usepackage{amsmath}
\usepackage{amscd}
\usepackage{mathtools, bm}
\usepackage{color}
\usepackage{enumerate}
\usepackage{dsfont}
\usepackage[latin1]{inputenc}
\usepackage[displaymath,mathlines]{lineno}
\usepackage[mathscr]{eucal}
\numberwithin{equation}{section}
\usepackage{epstopdf} 
\usepackage{hyperref}
\usepackage{indentfirst}
\usepackage{lineno}
\setpagewiselinenumbers
 
 \makeindex

\theoremstyle{plain}
\newtheorem{theorem}{Theorem}[section]
\newtheorem{lemma}[theorem]{Lemma}

\newtheorem{proposition}[theorem]{Proposition}

\theoremstyle{definition}
\newtheorem{definition}[theorem]{Definition}

\newtheorem{remark}[theorem]{Remark}

\title[]{Weak-strong uniqueness and high-friction limit \\ for Euler-Riesz systems}

\author[Nuno J. Alves, Jos\'{e} A. Carrillo \& Young-Pil Choi]{}
\address[Nuno J. Alves]{University of Vienna, Faculty of Mathematics, Oskar-Morgenstern-Platz 1, 1090 Vienna, Austria.}
\email{nuno.januario.alves@univie.ac.at}

\address[Jos\'{e} A. Carrillo]{University of Oxford, Mathematical Institute, Oxford, OX2 6GG, United Kingdom.}
\email{carrillo@maths.ox.ac.uk}

\address[Young-Pil Choi]{Yonsei University, Department of Mathematics, Seoul, 03722, Republic of Korea}
 
\email{ypchoi@yonsei.ac.kr}


\begin{document}
\maketitle

\centerline{\scshape Nuno J. Alves, Jos\'{e} A. Carrillo \& Young-Pil Choi}
\medskip

\date{\today}

\begin{abstract}
In this work we employ the relative energy method to obtain a weak-strong uniqueness principle for a Euler-Riesz system, as well as to establish its convergence in the high-friction limit towards a gradient flow equation. The main technical challenge in our analysis is addressed using a specific case of a Hardy-Littlewood-Sobolev inequality for Riesz potentials.
\end{abstract}

\maketitle

\baselineskip 18pt

\section{Introduction}
In this work we consider the following Euler-Riesz system in $\mathopen{]}0,T\mathclose{[}  \times \Omega$: 
\begin{equation} \label{EW}
 \begin{dcases}
  \partial_t \rho + \nabla \cdot (\rho u) = 0, \\
  \partial_t (\rho u) + \nabla \cdot (\rho u \otimes u) + \nabla \rho^\gamma + \kappa \rho \nabla W* \rho = - \nu \rho u,\\
 \end{dcases}
\end{equation} 
where $0 < T < \infty$ and $\Omega$ is either the $d$-dimensional torus $\mathbb{T}^d$ or a smooth bounded domain of $\mathbb{R}^d$, $d \geq 1$. In the case of a bounded domain, we consider the following no-flux boundary conditions:
\begin{equation} \label{BCEW}
u \cdot n = 0 \quad \text{on} \ \mathopen{[}0, T\mathclose{[} \times \partial \Omega
\end{equation}
where $n$ is any outward normal vector to the boundary of $ \Omega.$ The density and linear velocity are denoted by $\rho$  and $u,$ respectively, $\rho^\gamma= p(\rho)$ is the pressure, and the interaction kernel $W$ is given by 
\[
W(x) = \pm \frac{|x|^\alpha}{\alpha}, \quad -d < \alpha < 0.
\]
The parameters $\gamma > 1, \ \kappa > 0,$  $\nu \geq 0$ are the adiabatic exponent, the interaction strength, and the collision frequency, respectively. This system models a single-species fluid subject to attractive/repulsive interaction forces depending on the sign $\pm$.
\par
The goal of this work is twofold. First, one obtains a weak-strong uniqueness property for system $(\ref{EW})$, and then one establishes the high-friction limit ($\nu \to \infty$) of system $(\ref{EW})$ towards the following diffusion-aggregation equation:
\begin{equation} \label{GF}
\partial_t \rho = \nabla \cdot \big(\nabla \rho^\gamma + \kappa \rho\nabla W* \rho  \big) \quad \text{in} \ \mathopen{]}0,T\mathclose{[}  \times \Omega,
\end{equation}
with the boundary conditions (if $\Omega$ is a bounded domain),
\begin{equation} \label{BCGF}
\big(\nabla \rho^\gamma + \kappa \rho\nabla W* \rho  \big) \cdot n = 0  \quad \text{on} \ \mathopen{[}0, T\mathclose{[} \times \partial \Omega.
\end{equation}
These equations find applications in mathematical biology or plasma physics, for instance in the modelling of the behaviour of cell populations as adhesion or chemotaxis, or in the modelling of charged particles subject to electric forces, see \cite{calvez,carrillo, carrillochoi} and references therein. \par 
Both results are obtained using the relative energy method. In both cases, the necessary estimates are essentially identical. This 
serves as a clear illustration of the close mathematical relationship between weak-strong uniqueness principles and relaxation phenomena.   \par 
The relative energy method is an efficient methodology for achieving stability results, including weak-strong uniqueness principles, and establishing asymptotic limits. Since its origin, \textit{e.g.} \cite{dafermos}, this method has seen an extensive applicability to diffusive relaxation \cite{alves, bianchini, carrillorelativeentropy, lattanzio, lattanziogas}.  \par
In the present manuscript, we consider weak and strong solutions for (\ref{EW}) and strong solutions for (\ref{GF}). Equation (\ref{GF}) can be regarded as a gradient flow in the space of probability measures endowed with a Wasserstein distance \cite{carrilloGF1,carrilloGF}. In fact, equation (\ref{GF}) can be written as 
\begin{equation*}
 \begin{dcases}
  \partial_t \rho + \nabla \cdot (\rho u) = 0, \\
 \rho u= -\nabla \rho^\gamma - \kappa \rho \nabla W* \rho .\\
 \end{dcases}
\end{equation*} 
As a consequence, the velocity field is given by $u=-\nabla\frac{\delta \mathcal{E}}{\delta \rho}$ with the free energy functional $\mathcal{E}(\rho)$ defined as
\[
\mathcal{E}(\rho) = 
\int_\Omega  \tfrac{1}{\gamma-1} \rho^{\gamma} + \kappa \tfrac{1}{2} \rho (W * \rho) \, dx
\]
over probability densities and its variation computed with zero mass perturbations. Within that framework, the emergence of a gradient flow type equation as the high-friction limit of its Euler counterpart has also been studied \cite{carrillochoi}. The results presented here extend the ones obtained in \cite{alvesrole} and \cite{lattanziogas} to what concerns the weak-strong uniqueness principle and the high-friction limit, respectively. In the former, one establishes the weak-strong uniqueness principle for a Euler-Poisson system in the whole space, taking as solution of the Poisson equation the convolution of the density with the Newtonian kernel. The same result can be obtained in the bounded setting considering the Neumann function instead, which essentially corresponds to the present case with $\alpha = 2 - d$ and $d \geq 3$. In \cite{lattanziogas}, the authors establish the high-friction limit of a Euler-Poisson system towards a Keller-Segel system in a periodic setting. In this case, the high-friction limit is established for $\gamma \geq 2 - \frac{2}{d}$.  This range for the adiabatic exponent is also valid for a weak-strong uniqueness principle or high-friction limit of a Euler-Poisson system in a bounded domain, where the Neumann function is put together with the regularity theory employed in \cite{lattanziogas}. In those cases, using the Poisson equation one derives integration by parts formulas that prove very useful in obtaining the necessary estimates to reach the desired stability results. In the present case, those integration by parts formulas do not apply. Moreover, the same range for the adiabatic exponent was obtained in \cite{carrillorelativeentropy} where the authors establish the high-friction limit of a Euler system with a bounded interaction kernel. \par 
The high-friction limit of a pressureless Euler-Riesz system towards an aggregation equation is studied in \cite{choi0, choi2}. In the pressureless case, the relative interaction energy cannot be controlled by a pressure, and therefore should be controlled by either itself or by the relative kinetic energy. In \cite{choi0, choi2}, the Wasserstein distance combined with the relative kinetic energy is employed to handle the nonlocal interaction term when it is regular. On the other hand, the interaction potential $W$ is given by the Coulombic or super-Coulombic interaction and the relative interaction energy is used to bound the nonlocal terms. See also \cite{choi3, NRS22} for estimates of relative interaction energies with Riesz potentials. Additionally, theories on well-posedness and existence of solutions for Euler-Riesz systems can be found in \cite{choi1, danchin}.    \par 
The results obtained here are accomplished by means of the relative energy for (\ref{EW}), which is the quadratic part of the Taylor expansion of the energy functional, see Section \ref{section_energy}. For the weak-strong uniqueness principle, we compare weak and strong solutions of (\ref{EW}), while for the high-friction limit we compare a strong solution of $(\ref{GF})$ with a weak solution of (\ref{EW}) in the limit $\nu \to \infty$. The first step is to obtain an inequality satisfied by the relative energy between the two considered solutions. The terms on the right-hand-side of that inequality are then bounded using the relative energy. Among these estimates, the one deserving special attention is the one containing the interaction kernel, which is dealt with using a particular case of a Hardy-Littlewood-Sobolev inequality.

\par 
As a summary, we are able to prove our main results -- weak-strong uniqueness and high-friction limit -- in the ranges:
\[ \gamma \geq 2-\frac{\alpha+d}{d}, \quad  1-d < \alpha < 0, \quad d > 1.  \]
 
\par 
The weak-strong uniqueness result in stated in Section \ref{section_main_results_WSU} and it is proved in Section \ref{section_proofs_WSU}. The result concerning the high-friction limit is stated in Section \ref{section_main_results_RL} and its proof is sketched in Section \ref{section_proofs_RL}. Moreover, we have to restrict the admissible values of the interaction strength to ensure that the relative energy is non-negative, otherwise the stability results could not be obtained, see Section \ref{section_proofs_nonnegativity}. 

\begin{remark}
The results obtained in this manuscript may be understood as formal results as one considers notions of solutions whose existence theories are not yet available. Nonetheless, the weak solutions considered here are reasonable to what concerns the a priori estimates of system (\ref{EW}), and the strong solutions can be seen as classical solutions with additional boundedness assumptions. 
\end{remark}

\section{Energy and Relative Energy Identities} \label{section_energy}

In this section we formally derive the energy and relative energy identities for system \eqref{EW}. First, we introduce the internal energy function $h(\rho) = \frac{1}{\gamma - 1} \rho^\gamma,$ which is connected to the pressure $p(\rho) = \rho^\gamma$ via the relation
$\rho h^{\prime\prime}(\rho) = p^\prime(\rho)$, with $\rho > 0$. Moreover, the potential energy functional $\mathcal{E} = \mathcal{E}(\rho)$ associated with system \eqref{EW} is given by
\[
\mathcal{E}(\rho) = \int_\Omega  h(\rho) + \kappa \tfrac{1}{2} \rho (W * \rho)  \, dx.
\]

Using the symmetry of $W,$ one computes the functional derivative $\frac{\delta \mathcal{E}}{\delta \rho},$ which is given by
\[
\frac{\delta \mathcal{E}}{\delta \rho}(\rho) = h^\prime(\rho) + \kappa (W * \rho).
\]

The relative potential energy functional associated with \eqref{EW} is then given by
\begin{equation*} \label{relpotential}
\begin{split} 
\mathcal{E}(\rho | \bar \rho) &= \mathcal{E}(\rho) - \mathcal{E}(\bar \rho) - \big\langle \frac{\delta \mathcal{E}}{\delta \rho}(\bar \rho), \rho - \bar \rho \big\rangle  \\
&= \int_{\Omega}  h(\rho | \bar \rho) + \kappa \tfrac{1}{2} (\rho - \bar \rho) (W * (\rho - \bar \rho))  \, dx,
\end{split}
\end{equation*}
where $h(\rho | \bar \rho) = h(\rho) - h(\bar \rho) - h^\prime(\bar \rho)(\rho - \bar \rho).$ \par 
Assume that $(\rho, u)$ is a smooth solution of (\ref{EW}) (satisfying (\ref{BCEW}) if $\Omega$ is a bounded domain). By taking the inner product of the second equation in (\ref{EW}) with $u$, integrating the resulting expression over space, and utilizing the symmetry of $W$, we obtain the energy identity for system (\ref{EW}):
\begin{equation*} \label{energyidentityEW1}
\frac{d}{dt} \int_\Omega \tfrac{1}{2} \rho |u|^2 + h(\rho) + \kappa \tfrac{1}{2} \rho (W * \rho)  \, dx = - \int_\Omega \nu \rho |u|^2 \, dx.
\end{equation*}
Observe that for $\nu=0$ the energy is conserved over time while for $\nu>0$ the energy is dissipated. \par 
The kinetic energy functional associated with system (\ref{EW}) is the functional $\mathcal{K}=\mathcal{K}(\rho, \rho u)$ given by 
\[\mathcal{K}(\rho, \rho u) = \int_\Omega \tfrac{1}{2} \rho |u|^2 \, dx. \]
After a straightforward derivation we obtain the following functional derivatives: 
\[
\frac{\delta \mathcal{K}}{\delta \rho}(\rho, \rho u) = -\tfrac{1}{2}|u|^2, \qquad \frac{\delta \mathcal{K}}{\delta (\rho u)}(\rho, \rho u) = u.
\]  
Using these derivatives we can calculate the relative kinetic energy functional associated with (\ref{EW}):
\begin{equation*} \label{relkinetic1}
   \begin{split}
    \mathcal{K} (\rho,  \rho u  |  \bar{\rho}, \bar{\rho} \bar{u})  & =  \ \mathcal{K}(\rho, \rho u) - \mathcal{K}(\bar{\rho}, \bar{\rho} \bar{u})  - \big< \frac{\delta \mathcal{K}}{\delta \rho}(\bar{\rho}, \bar{\rho} \bar{u}) , \rho - \bar{\rho}\big> - \big< \frac{\delta \mathcal{K}}{\delta (\rho u)}(\bar{\rho}, \bar{\rho} \bar{u}) , \rho u - \bar{\rho} \bar u\big> \\
    & =   \int_{\Omega} \tfrac{1}{2} \rho |u - \bar{u}|^2  \, dx.
    \end{split}
 \end{equation*} 
Now, we present the relative energy identity for system (\ref{EW}). Assuming that $(\bar \rho , \bar u)$ is another smooth solution of (\ref{EW}), by adding the relative kinetic and relative potential energies of (\ref{EW}), after taking the time derivative we obtain:
\begin{align*}
& \frac{d}{dt} \int_\Omega \tfrac{1}{2} \rho |u - \bar u|^2 + h(\rho| \bar \rho) + \kappa \tfrac{1}{2} (\rho - \bar \rho) \big(W * (\rho - \bar \rho) \big)  \, dx  + \int_\Omega \nu \rho |u - \bar u|^2 \, dx \\ 
 = &   - \int_\Omega \nabla \bar u : \rho (u - \bar u) \otimes (u - \bar u) \, dx \\
& - \int_\Omega (\nabla \cdot \bar u )p(\rho | \bar \rho) \ dx \\
& + \int_\Omega \kappa (\rho - \bar \rho) \bar u \cdot \nabla W * (\rho - \bar \rho) \, dx.
\end{align*}

\section{Main results} \label{section_main_results}
In this section we present the main results of this work -- a weak-strong uniqueness property for system 
(\ref{EW}), and its high-friction limit towards (\ref{GF}) as $\nu \to \infty$.  \par  
Before stating the results for each case, we provide the notions of solutions that will be employed in our analysis. Specifically, for the weak-strong uniqueness property we consider weak and strong solutions of system (\ref{EW}), while for the high-friction limit we focus on weak 
solutions of system (\ref{EW}) and strong solutions of system 
(\ref{GF}).

\subsection{Weak-strong uniqueness} \label{section_main_results_WSU}~\par 
\begin{definition}
 A pair $(\rho, u)$ with non-negative $\rho$ and such that
 \[\rho \in C\big([0,T[;  L^\gamma(\Omega) \big),\]
 \[\rho u \in C\big([0,T[;L^1(\Omega,\mathbb{R}^d)\big), \] 
 \[\rho |u|^2 \in  C\big([0,T[;L^1(\Omega)\big), \]
is a weak solution of (\ref{EW}) if:
\begin{enumerate}[(i)]
 \item $(\rho, u)$ satisfies a weak form of (\ref{EW}):
 \begin{equation*} 
         -\int_0^T \int_{\Omega} (\partial_t \varphi) \rho \, dxdt-\int_0^T \int_{\Omega} \nabla \varphi \cdot (\rho u) \, dxdt - \int_{\Omega} \varphi(0,x) \rho_0(x) \, dx=0,
\end{equation*} 
        \vspace{1mm}
 \begin{equation*} 
        \begin{split}
         -   \int_0^T & \int_{\Omega} \partial_t \tilde{\varphi} \cdot (\rho u) \, dxdt -  \int_0^T \int_{\Omega} \nabla \tilde{\varphi} : \rho u \otimes u \, dx dt - \int_{\Omega} \tilde{\varphi}(0,x) \cdot \rho_0(x) u_0(x) \, dx\\
         - & \int_0^T \int_{\Omega} (\nabla \cdot \tilde{\varphi}) \rho^\gamma \, dxdt   -  \int_0^T \int_{\Omega} \tilde{\varphi}\cdot (\kappa \rho \nabla W * \rho) \, dxdt = - \int_0^T \int_{\Omega} \nu \tilde{\varphi}\cdot (\rho u)\, dxdt,
        \end{split}
\end{equation*} 
for any Lipschitz test functions $\varphi: \mathopen{[}0,T\mathclose{[} \times \Omega \to \mathbb{R}$ and $ \tilde{\varphi}:\mathopen{[}0,T\mathclose{[} \times \Omega \to \mathbb{R}^d$ that have compact support in time, and in the case of a bounded domain, satisfy $\tilde{\varphi} \cdot n = 0$ on $\mathopen{[}0,T\mathclose{[} \times \partial \Omega,$ where $n$ denotes any outward normal vector to the boundary $\partial \Omega,$
    \item the mass is conserved:
    \begin{equation*} 
 \int_\Omega \rho(t,x) \, dx = M < \infty  \quad \text{for} \ t \in \mathopen{[}0,T\mathclose{[},
  \end{equation*} 
  \item the energy is finite:
  \begin{equation*} 
  \underset{[0,T[}{\sup} \int_\Omega  \tfrac{1}{2} \rho |u|^2 + h(\rho) + \kappa \tfrac{1}{2} \rho (W * \rho) \, dx < \infty.
  \end{equation*}
\end{enumerate}
Furthermore, a weak solution of (\ref{EW}) is called dissipative if $\rho (W * \rho) \in C\big([0,T[;  L^1(\Omega) \big)$ and
 \begin{equation} \label{weakdissipEW1}  
 \begin{split}
-  \int_0^T \int_\Omega & \big( \tfrac{1}{2} \rho |u|^2 + h(\rho) + \kappa \tfrac{1}{2} \rho (W * \rho) \big) \dot{\theta}(t) \, dxdt + \int_0^T \int_\Omega \nu \rho |u|^2 \theta(t) \, dxdt \\
    \leq  \int_\Omega &  \big( \tfrac{1}{2} \rho |u|^2 + h(\rho) + \kappa \tfrac{1}{2} \rho (W * \rho) \big) \big|_{t=0}  \theta(0) \, dx
    \end{split}
   \end{equation} 
  for all non-negative and compactly supported $\theta \in W^{1,\infty}([0,T[)$.
  \end{definition}
\vspace{2mm} \par 
Now, we present the notion of strong solution for (\ref{EW}). 
\begin{definition}
A Lipschitz pair $(\bar \rho, \bar u),$ with positive $\bar \rho$ and such that 
 \[\bar \rho \in L^\infty\big([0,T[;  L^\infty(\Omega) \big),\]
 \[\bar u \in L^\infty\big([0,T[;W^{1,\infty}(\Omega,\mathbb{R}^d)\big), \]
is a strong solution of (\ref{EW}) if:
\begin{enumerate}[(i)]
\item $(\bar \rho, \bar u)$ satisfies 
\begin{equation} \label{strongEW1}
\begin{cases}
\partial_t  \bar \rho + \nabla \cdot (\bar \rho \bar u) = 0, \\
\partial_t \bar u + (\bar u \cdot \nabla) \bar u + \nabla h^\prime(\bar \rho) + \kappa \nabla W * \bar \rho = -\nu \bar \rho \bar u,
\end{cases}
\end{equation}
for a.e. $(t,x) \in \mathopen{]}0, T\mathclose{[} \times \Omega,$
\item the functions $\varphi$, $\tilde{\varphi}$ given by
\[\varphi = \theta \big(-\tfrac{1}{2}|\bar u|^2 + h^\prime(\bar \rho) + \kappa W * \bar \rho \big),\] \[\tilde{\varphi} = \theta \bar u,\]
can be used as test functions in the weak formulation detailed above, where $\theta$ is given by (\ref{theta}),
\item $(\bar \rho, \bar u)$ satisfies the boundary conditions (\ref{BCEW}) in the case of a bounded domain. 
\end{enumerate}
\end{definition}
Also, we prescribe initial data $(\bar \rho_0, \bar u_0)$ that satisfies 
\begin{equation*} 
\int_\Omega \bar \rho_0 \, dx = \bar M < \infty
\qquad \mbox{and} \qquad
 \int_\Omega \tfrac{1}{2}\bar \rho_0 |\bar u_0|^2 + h(\bar \rho_0) + \kappa \tfrac{1}{2} \bar \rho_0 (W * \bar \rho_0) \, dx < \infty.
\end{equation*}

A feature of a strong solutions is that it satisfies the energy identity. This can be seen by multiplying the second equation of (\ref{strongEW1}) by $\bar \rho \bar u$ and then integrating it over space, leading to:
\begin{equation} \label{strongenergyEW1} 
\frac{d}{dt} \int_\Omega \tfrac{1}{2}\bar \rho |\bar u|^2+h( \bar \rho) + \kappa \tfrac{1}{2} \bar \rho (W * \bar \rho) \, dx = - \int_\Omega \nu \bar \rho |\bar u|^2 \, dx.
\end{equation}

Furthermore, we consider strong solutions $(\bar \rho, \bar u)$ of (\ref{EW}) that are bounded away from vacuum, that is, there exist $\hat{\delta} > 0$ and $\hat{M} < \infty$ such that 
\begin{equation} \label{boundawayzero}
\bar \rho(t,x) \in [\hat{\delta},\hat{M}]  \quad \text{for} \ (t,x) \in [0,T[.
\end{equation}
Weak and strong solutions of (\ref{EW}) are compared using the relative energy $\Psi : [0, T[ \to \mathbb{R}$ defined as follows:
\[\Psi(t) = \int_\Omega \tfrac{1}{2} \rho |u - \bar u|^2 + h(\rho| \bar \rho) + \kappa \tfrac{1}{2} (\rho - \bar \rho) \big(W * (\rho - \bar \rho) \big) \, dx. \]
For the notions of solutions described above, we obtain the following weak-strong uniqueness result:
\begin{theorem} \label{thmwsu}
Let $d>1$ and let $(\rho, u)$ be a dissipative weak solution of (\ref{EW}) with $1-d < \alpha < 0,$ $\gamma \geq 2 - \frac{\alpha + d}{d}$ and $\kappa$ sufficiently small. Let $(\bar \rho, \bar u)$ be a strong and bounded away from vacuum solution of (\ref{EW}). Then, there exists $C>0$ such that $\Psi$ satisfies 
\begin{equation} \label{stability1}
\underset{[0,T[}{\sup} \ \Psi \leq e^{CT}\Psi(0).
\end{equation}
Therefore, if $(\rho_0, u_0) = (\bar \rho_0, \bar u_0)$ then
\[(\rho(t), u(t)) = (\bar \rho(t), \bar u(t))  \quad \text{for} \ t \in [0,T[. \]
\end{theorem}
\vspace{1mm}
\subsection{High-friction limit towards a gradient flow} \label{section_main_results_RL}~\par
Taking $\nu \to \infty$ in (\ref{EW}) formally yields the equilibrium state $u=0$. Therefore, in order to capture the limiting phenomena we must first consider a diffusive scaling. Let $\varepsilon = 1/ \nu^2 $ and consider the scaling 
\[\tilde{\rho}(t,x) = \rho(\varepsilon^{-1/2}t,x),\] 
\[\tilde{u}(t,x) = \varepsilon^{-1/2} u (\varepsilon^{-1/2}t,x). \]
Dropping the tilde notation gives
\begin{equation} \label{EW2}
 \begin{dcases}
  \partial_t \rho + \nabla \cdot (\rho u) = 0, \\
  \varepsilon \big( \partial_t (\rho u) + \nabla \cdot (\rho u \otimes u) \big) + \nabla \rho^\gamma + \kappa \rho \nabla W* \rho = - \rho u. \ 
 \end{dcases}
\end{equation} 

Taking $\varepsilon \to 0$ in (\ref{EW2}), one formally obtains (\ref{GF}). \par 
We consider dissipative weak solutions of (\ref{EW2}), which are completely analogous to the dissipative weak solutions of (\ref{EW}) described in the previous section. \par
Now, we present the notion of strong solution of (\ref{GF}), to which one can establish the convergence in the high-friction limit.  \par 
\begin{definition}
A positive Lipschitz function $\bar \rho,$ such that 
 \[\bar \rho \in W^{1, \infty}\big([0,T[;  W^{1,\infty}(\Omega) \big), \]
 is a strong solution of (\ref{GF}) if:
\begin{enumerate}[(i)]
\item $\bar \rho$ satisfies 
\begin{equation} \label{strongGF}
\partial_t  \bar \rho = \nabla \cdot \big( \bar \rho \nabla(h^\prime(\bar \rho)+ \kappa W * \bar \rho) \big) \\
\end{equation}
for a.e. $(t,x) \in \mathopen{]}0, T\mathclose{[} \times \Omega,$
\item the functions $\varphi: \mathopen{[}0,T\mathclose{[} \times \Omega \to \mathbb{R}, \ \tilde{\varphi}:\mathopen{[}0,T\mathclose{[} \times \Omega \to \mathbb{R}^d$ given by 
\[\varphi = \theta \big(-\tfrac{1}{2}| \nabla h^\prime(\bar \rho) + \kappa \nabla W * \bar \rho|^2 + h^\prime(\bar \rho) + \kappa W * \bar \rho \big), \] 
\[\tilde{\varphi} = \theta (- \nabla h^\prime(\bar \rho) - \kappa \nabla W * \bar \rho ),\]
can be used as test functions in weak formulation detailed above, where $\theta$ is given by (\ref{theta}),
\item $\bar \rho$ satisfies the boundary conditions (\ref{BCGF}) in the case of a bounded domain,
\item Additionally, 
\[ \frac{\partial^2 \bar \rho}{\partial x_i \partial x_j}, \ \frac{\partial^2 \bar \rho}{\partial x_i \partial t}, \ \frac{\partial^2 (W * \bar \rho)}{\partial x_i \partial x_j}, \ \frac{\partial^2 (W * \bar \rho)}{\partial x_i \partial t} \in L^\infty(\mathopen{[}0,T\mathclose{[} \times \Omega)  \]
for $i,j = 1, \ldots, d.$
\end{enumerate}
 \par 
\end{definition}
We prescribe initial data $\bar \rho_0$ that satisfies 
\begin{equation*} 
\int_\Omega \bar \rho_0 \, dx = \bar M < \infty,\qquad \mbox{and} \qquad
 \int_\Omega h(\bar \rho_0) + \kappa \tfrac{1}{2} \bar \rho_0 (W * \bar \rho_0) \, dx < \infty.
\end{equation*}
\par 
To compare a solution $(\rho, u)$ of (\ref{EW2}) with a solution $\bar \rho$ of (\ref{GF}), we view the solution of (\ref{GF}) as an approximation to (\ref{EW2}) using the following equations:
\begin{equation*} \label{approximateEW2}
 \begin{dcases}
  \partial_t \bar \rho + \nabla \cdot (\bar \rho \bar u) = 0, \\
  \varepsilon \big( \partial_t (\bar \rho \bar u) + \nabla \cdot (\bar \rho \bar u \otimes \bar u) \big) + \nabla p(\bar \rho) + \kappa \bar \rho \nabla W* \bar \rho = - \bar \rho \bar u + \varepsilon \bar e. \ 
 \end{dcases}
\end{equation*} 
Here, $\bar u$ is defined as $\bar u = - \nabla h^\prime(\bar \rho) - \kappa \nabla W * \bar \rho,$ and $\bar e$ is given as 
$\bar e =   \partial_t (\bar \rho \bar u) + \nabla \cdot (\bar \rho \bar u \otimes \bar u).$
Under the specified regularity conditions, there exists a positive constant $C$ such that 
\[\|\bar u \|_\infty, \| \nabla \bar u \|_\infty, \| \bar e \|_\infty \leq C. \] \par

These strong solutions dissipate the total energy. Indeed, multiplying the equation $\bar u = - \nabla h^\prime(\bar \rho) - \kappa \nabla W * \bar \rho$ by $\bar \rho \bar u$, then using equation (\ref{strongGF}) and integrating it over space leads to:
\begin{equation*} 
\frac{d}{dt} \int_\Omega h(\bar \rho) + \kappa \tfrac{1}{2} \bar \rho (W * \bar \rho) \, dx = - \int_\Omega \bar \rho |\bar u|^2 \, dx.
\end{equation*}

In this case, we consider the relative energy $\Phi_\varepsilon: [0, T[ \to \mathbb{R}$ defined as
\[\Phi_\varepsilon(t) = \int_\Omega \varepsilon \tfrac{1}{2} \rho |u - \bar u|^2 + h(\rho| \bar \rho) + \kappa \tfrac{1}{2} (\rho - \bar \rho) \big(W * (\rho - \bar \rho) \big) \, dx. \]
The convergence in the high-friction limit of a weak solution of (\ref{EW2}) towards a strong solution of (\ref{GF}) is established in the following result:
\begin{theorem} \label{thmrelax}
Let $d >1$ and let $(\rho, u)$ be a dissipative weak solution of (\ref{EW2}) with $1-d < \alpha < 0,$ $\gamma \geq 2 - \frac{\alpha + d }{d}$ and $\kappa$ sufficiently small. Let $\bar \rho$ be a strong solution of (\ref{GF}) satisfying (\ref{boundawayzero}). Then, there exists $C>0$ such that $\Phi$ satisfies
\begin{equation} \label{stability2}
\underset{[0,T[}{\sup} \ \Phi_\varepsilon \leq e^{CT}(\Phi_\varepsilon(0) + \varepsilon^2) .
\end{equation} 
Thus, if $\Phi_\varepsilon(0) \to 0$ as $\varepsilon \to 0,$ then 
\[\underset{[0,T[}{\sup} \ \Phi_\varepsilon \to 0 \quad \text{as} \ \varepsilon \to 0.\]

\end{theorem}

\section{Proofs} \label{section_proofs}
This section contains the proofs of Theorem \ref{thmwsu} and Theorem \ref{thmrelax}.
\subsection{Auxiliary results} \label{section_proofs_auxiliary}
First, we recall the notion of Riesz potentials. Given a measurable function $f: \Omega \to \mathbb{R}$ and $0 < \beta < d,$ the Riesz potential of degree $\beta$ of $f$ is the function $I_\beta f : \Omega \to \mathbb{R}$ given by \[I_\beta f(x) = \int_\Omega \frac{f(y)}{|x-y|^{d-\beta}} \, dy.\]
Furthermore, if $f \in L^q(\Omega)$ with $1 < q < d/\beta,$ then $I_\beta f$ enjoys the following integrability estimate \cite{hedberg}:
\begin{equation} \label{stein}
\|I_\beta f \|_{\frac{dq}{d-\beta q}} \leq C(d,\beta,q) \ \|f\|_q .
\end{equation}
Using this terminology, the potential $W * \rho$ can be rewritten as 
\[
W * \rho = \pm \tfrac{1}{d-\beta}I_\beta \rho \ , 
\]
where $\beta = \alpha + d, \ 0 < \beta < d. $ \par 
The next result aims to give meaning to the gradient of the potential $W \ast \rho$, following the same ideas of \cite[Proposition 3.1]{alvesrole}.
\begin{proposition}
Let $d > 1$ and let $\rho \in C\big([0,T[;  L^\gamma(\Omega) \big)$ with $\gamma \geq \gamma_0 = 2 - \tfrac{\beta}{d}, \ 1 < \beta < d.$ Then $W * \rho \in C\big([0,T[;  L^r(\Omega) \big),$ where $r = \frac{(2d-\beta)d}{(d-\beta)^2}$, and its spatial gradient $\nabla(W * \rho)$ is given by
\[\nabla(W * \rho)(t,x) = \nabla W * \rho(t,x) = \pm \int_\Omega \frac{(x-y)}{|x-y|^{d+2-\beta}} \rho(t,y) \, dy  ,\]  
from which \[|\nabla W * \rho | \leq I_{\beta - 1} |\rho|  . \]
\end{proposition}
 \begin{proof}
We prove the first assertion using (\ref{stein}). Under the current hypotheses, $\rho \in C\big([0,T[;  L^{\gamma_0}(\Omega) \big)$. A straightforward calculation gives:
\[ \frac{d\gamma_0}{d- \beta \gamma_0} = \frac{(2d-\beta)d}{(d-\beta)^2} = r > 1 , \]
and $\gamma_0 < \frac{d}{\beta}$.
Therefore, 
\[\|W * \rho \|_r = \tfrac{1}{d - \beta}\| I_\beta \rho\|_{\frac{d\gamma_0}{d- \beta \gamma_0}} \leq C \, \| \rho\|_{\gamma_0}. \] The linearity of $I_\beta$ then implies that $W * \rho \in C\big([0,T[;  L^1(\Omega) \big),$ as desired. \par
Now, we proceed to prove the second assertion. The condition $1 < \beta$ is required in this part. Take an arbitrary $v \in C_c^\infty(\Omega).$ We aim to prove that 
\[ \int_\Omega (W * \rho) v_{x_i} \, dx = - \int_\Omega (W_{x_i} * \rho) v \, dx \ , \quad i = 1, \ldots, d. \] 
Since $W * \rho \in L^1(\Omega),$ we have $\int_\Omega (W * \rho) v_{x_i} \ dx < \infty,$ which, by Fubini's theorem, implies that
\[\int_\Omega (W*\rho) v_{x_i} \, dx = \int_\Omega \int_\Omega W(x-y) v_{x_i}(x) \, dx \, \rho \, dy. \] 
Let $B_\varepsilon(y)$ be the open ball in $\Omega$ with center $y$ and radius $\varepsilon > 0.$ Note that the map $x \to W(x - y)$ is smooth in $B^c_\varepsilon(y) = \Omega \setminus B_\varepsilon(y),$ and hence 
\begin{equation*} 
\begin{split}
\int_{\Omega}  W(x-y) v_{x_i}(x) \, dx = & \int_{B_\varepsilon(y)}  W(x-y) v_{x_i}(x) \, dx + \int_{B_\varepsilon^c(y)}  W(x-y) v_{x_i}(x) \, dx \\
=& \int_{B_\varepsilon(y)}  W(x-y) v_{x_i}(x) \, dx + \int_{\partial B_\varepsilon(y)}  W(x-y) v(x) n_i(x) \, dS(x) \\
& - \int_{B_\varepsilon^c(y)}  W_{x_i}(x-y) v(x) \, dx,
\end{split}
\end{equation*} 
 where $n_i$ is the $i$-th component of the unit inward normal vector to the boundary of $B_\varepsilon(y).$ Using polar coordinates we derive:
\begin{equation*}
 \begin{split}
 \Big|\int_{B_\varepsilon(y)}  W(x-y) v_{x_i}(x) \, dx \Big| & \leq C \int_{B_\varepsilon(y)} \frac{1}{|x-y|^{d-\beta}} \, dx \\
                 & = C \int_0^\varepsilon \int_{\partial B_r(y)} \frac{1}{|x-y|^{d-\beta}} \, dS(x) dr = C \varepsilon^\beta ,         
 \end{split}
 \end{equation*}
and
 \begin{equation*}
 \begin{split}
\Big|\int_{\partial B_\varepsilon(y)}  W(x-y) v(x) n_i(x) \ dS(x)\Big| & \leq C \int_{\partial B_\varepsilon(y)} \frac{1}{|x-y|^{d-\beta}} \, dS(x) = C \varepsilon^{\beta - 1}  .
 \end{split}
 \end{equation*}  
The desired identity is then obtained by letting $\varepsilon \to 0$ above. 
 \end{proof}
 
Next, we state a lemma that combines a particular case of the Hardy-Littlewood-Sobolev (HLS) inequality with interpolation of Lebesgue spaces. For the HLS inequality, we refer to \cite[Theorem 4.3]{lieb}.
\begin{lemma} \label{HLSlemma}
Let $0 < \beta < d,$ $p = \tfrac{2d}{d+\beta},$ and $\gamma_0 = 2 - \tfrac{\beta}{d},$ so that $1 < p < \gamma_0 < 2.$ If $\rho \in L^\gamma(\Omega)$ with $\gamma \geq \gamma_0,$ then there exists $C>0$ such that 
\begin{equation} \label{HLS}
\| \rho I_\beta \rho \|_1 \leq C \, \| \rho \|_p^2 \leq C \, \|\rho \|_1^{\beta/d} \, \|\rho \|_{\gamma_0}^{\gamma_0}.
\end{equation}
\end{lemma}
The last result of this section, proved in \cite{lattanzio}, provides a way of bounding the relative function $h(\rho | \bar \rho)$ from below.
\begin{lemma} \label{hlemma}
Let $h(\rho) = \frac{1}{\gamma - 1} \rho^\gamma$ with $\gamma > 1,$ and let $\bar \rho \in [\delta, \bar M],$ where $\delta > 0$ and $\bar M < \infty.$ There exist $R \geq \bar M {+} 1$ and $C_1, C_2 > 0$ depending on $\gamma, \delta, \bar M$ such that 
\begin{equation*} 
 h(\rho| \bar \rho) \geq
 \begin{dcases}
  C_1 |\rho-\bar \rho|^2  \quad &\text{if} \ \rho \in  [0,R], \\
  C_2 |\rho-\bar \rho|^{\gamma}  \quad &\text{if} \ \rho \in \ ]R,\infty [.
\end{dcases}
\end{equation*}
\end{lemma}
\subsection{Non-negativity of the relative potential energy} \label{section_proofs_nonnegativity}
In this section we prove that for sufficiently small $\kappa$ the relative potential energy remains non-negative, which implies that the relative total energy also remains non-negative. This is essential for using the relative total energy as a yardstick for comparing solutions.
\begin{proposition} \label{Cstar}
Let $\rho \in L^\gamma(\Omega)$ with $\gamma \geq 2 - \frac{\alpha + d}{d}, \ -d < \alpha < 0,$ be non-negative, and let $\bar \rho \in L^\infty(\Omega)$ be such that $0 < \delta \leq \bar \rho \leq \bar M < \infty. $ There exists a positive constant $C_*$ such that 
\begin{equation}
\Big| \int_\Omega (\rho - \bar \rho) \big( W * (\rho - \bar \rho) \big) \ dx \Big| \leq C_* \int_\Omega h(\rho | \bar \rho) \, dx.
\end{equation}
\end{proposition}
\begin{proof}
Let $\beta = \alpha + d,$ $p = \tfrac{2d}{d+\beta},$ and $\gamma_0 = 2 - \tfrac{\beta}{d}.$ By the first inequality in Lemma \ref{HLSlemma} we have:
\[\Big| \int_\Omega (\rho - \bar \rho) \big( W * (\rho - \bar \rho) \big) \, dx \Big| \leq \tfrac{1}{d-\beta}\|(\rho - \bar \rho) I_\beta (\rho - \bar \rho) \|_1 \leq C \, \|\rho - \bar \rho \|_p^2. \]
Let $R$ be as in Lemma \ref{hlemma} and consider the sets $B = \{x \in \Omega \ | \ 0 \leq \rho(x) \leq R \}$, and its complement $U = \Omega \setminus B = \{x \in \Omega \ | \ \rho(x) > R \}$.
Then, the inclusion $L^2 \subseteq L^p$ and the second inequality in (\ref{HLS}) yield that:
\begin{equation*}
\begin{split}
\|\rho - \bar \rho \|_p^2 & \leq C \Big( \int_B |\rho - \bar \rho|^p \ dx  \Big)^{2/p} + C \Big( \int_U |\rho - \bar \rho|^p \, dx  \Big)^{2/p} \\
& \leq C \int_B |\rho - \bar \rho|^2 \ dx + C \int_U |\rho - \bar \rho|^{\gamma_0} \, dx.
\end{split}
\end{equation*}
By Lemma \ref{hlemma} it follows that 
\[ \int_B |\rho - \bar \rho|^2 \ dx \leq C \int_\Omega h(\rho | \bar \rho) \, dx.\]
Now, almost everywhere in $U$ it holds that $|\rho - \bar \rho| \geq 1$, hence, by Lemma \ref{hlemma} we obtain 
\[ \int_U |\rho - \bar \rho|^{\gamma_0} \ dx \leq \int_U |\rho - \bar \rho|^{\gamma} \ dx \leq C \int_\Omega h(\rho | \bar \rho) \, dx,\]
which completes the proof.
\end{proof} 
Under the same conditions of the previous result, we choose $\kappa$ so that 
$0 < \kappa < \frac{2}{C_*}$,
and set $\lambda \coloneqq 1 - \frac{\kappa C_*}{2} > 0$.
Thus, 
\[0 \leq \int_\Omega h(\rho | \bar \rho) \ dx \leq \frac{1}{\lambda} \int_\Omega h(\rho | \bar \rho) + \kappa \tfrac{1}{2}(\rho - \bar \rho) W * (\rho - \bar \rho) \, dx.  \]

\subsection{Proof of Theorem \ref{thmwsu}} \label{section_proofs_WSU}
The first step to establish the desired theorem is what is called the relative energy inequality. The calculations involved are by now standard and similar results can be found, for example, in \cite{lattanziogas, carrillorelativeentropy, alves}. Nevertheless, regarding the present case, we provide the details for completeness.  
\begin{proposition} \label{relineprop1}
Let $(\rho, u)$ be a dissipative weak solution of (\ref{EW}), and let $(\bar \rho, \bar u)$ be a strong solution of (\ref{EW}). Then, for each $t \in [0,T[,$ the relative energy $\Psi$ satisfies 
\begin{equation} \label{relativeinequality1}
\Psi(t) - \Psi(0) + \int_0^t \int_\Omega \nu \rho |u-\bar u|^2 \, dxdt \leq \mathcal{I}_1(t) + \mathcal{I}_2(t) + \mathcal{I}_3(t),
\end{equation}
where 
\begin{equation*}
\begin{split}
\mathcal{I}_1(t) & = - \int_0^t \int_\Omega \nabla \bar u : \rho (u - \bar u) \otimes (u - \bar u) \, dxd\tau, \\
\mathcal{I}_2(t) & = - \int_0^t \int_\Omega (\nabla \cdot \bar u) p(\rho | \bar \rho) \, dxd\tau, \\
\mathcal{I}_3(t) & = \int_0^t \int_\Omega \kappa (\rho - \bar \rho) \bar u \cdot \nabla W *(\rho - \bar \rho) \, dxd\tau.
\end{split}
\end{equation*}
\end{proposition}
\begin{proof}
Fix $t \in [0,T[,$  and let $\theta : [0,T[ \to \mathbb{R}$ be defined by
\begin{equation} \label{theta}
\theta(\tau) = 
\begin{cases}
1 &\text{if} \ 0 \leq \tau < t, \\
\dfrac{t-\tau}{k} + 1 &\text{if} \ t \leq \tau < t + k, \\
0 &\text{if} \ t+k \leq \tau < T,
\end{cases} 
\end{equation}
where $k>0$ is such that $t + k < T.$ Using this function $\theta$ in (\ref{weakdissipEW1}), after letting $\kappa \to 0$ we obtain:
\begin{equation} \label{weakenergy}
\int_\Omega \tfrac{1}{2}\rho|u|^2+h(\rho)+ \kappa \tfrac{1}{2} \rho (W* \rho)  \, dx \Big|_{\tau=0}^{\tau=t} \leq -\int_0^t \int_\Omega \nu \rho |u|^2 \, dxd\tau.
\end{equation}
Regarding the strong solution, integrating (\ref{strongenergyEW1}) over $]0,t[$ results in:
\begin{equation} \label{strongenergy}
\int_\Omega \tfrac{1}{2}\bar \rho|\bar u|^2+h(\bar \rho)+\kappa \tfrac{1}{2} \bar \rho (W * \bar \rho)  \, dx \Big|_{\tau=0}^{\tau=t} = -\int_0^t \int_\Omega \nu \bar \rho |\bar u|^2 \, dxd\tau.
\end{equation}
Next, we consider the weak formulation applied to the difference $(\rho - \bar \rho, \rho u - \bar \rho \bar u)$ between the weak and strong solutions. The weak formulation reads:
\begin{equation*}
         -\int_0^T \int_{\Omega} (\partial_t \varphi) (\rho - \bar{\rho}) \ dxdt-\int_0^T \int_{\Omega} \nabla \varphi \cdot (\rho u - \bar{\rho} \bar{u}) \ dxdt - \int_{\Omega} \varphi (\rho - \bar{\rho}) \big|_{t=0} \, dx=0,
\end{equation*}
\begin{align*}
        &-  \int_0^T \int_{\Omega} \partial_t \tilde{\varphi} \cdot (\rho u- \bar{\rho} \bar{u}) \ dxdt -  \int_0^T \int_{\Omega} \nabla \tilde{\varphi} : (\rho u \otimes u - \bar{\rho} \bar{u} \otimes \bar{u}) \, dx dt \\
        &-  \int_0^T \int_{\Omega} (\nabla \cdot \tilde{\varphi}) \big(p(\rho)-p(\bar{\rho}) \big) \, dxdt
        - \int_{\Omega} \tilde{\varphi} \cdot (\rho u - \bar{\rho} \bar{u})\big|_{t=0} \, dx \\
        =& \  -  \int_0^T \int_{\Omega} \tilde{\varphi} \cdot (\kappa \rho \nabla W*\rho - \kappa \bar{\rho} \nabla W*\bar \rho) \, dxdt - \int_0^T \int_{\Omega} \nu \tilde{\varphi} \cdot (\rho u - \bar \rho \bar u) \, dxdt,
\end{align*} 
for any Lipschitz test functions $\varphi: [0,T[ \times \Omega \to \mathbb{R}, \ \tilde{\varphi}:[0,T[ \times \Omega \to \mathbb{R}^d$ that have compact support in time, and satisfy $\tilde{\varphi} \cdot n = 0$ on $\mathopen{[}0,T\mathclose{[} \times \partial \Omega$.  

Consider as tests functions above the functions given by $\varphi = \theta \big(-\tfrac{1}{2}|\bar u|^2 + h^\prime(\bar \rho) + \kappa W* \bar \rho \big) $ and $\tilde{\varphi} = \theta \bar u$, 
where $\theta$ is the same as in (\ref{theta}). Letting $k \to 0$ results in:
\begin{align} \label{weakdiff1}
      & \int_\Omega  \big(- \tfrac{1}{2}|\bar{u}|^2+h^\prime(\bar{\rho})+ \kappa W * \bar \rho \big)(\rho - \bar{\rho})  \, dx \Big|_{\tau = 0}^{\tau = t} \nonumber \\
       & -\int_0^t \int_\Omega \partial_\tau \big(-  \tfrac{1}{2}|\bar{u}|^2+ h^\prime(\bar{\rho})+ \kappa W * \bar \rho \big)(\rho - \bar{\rho}) \, dxd\tau \\
     & -  \int_0^t \int_\Omega \nabla \big(- \tfrac{1}{2}|\bar{u}|^2+h^\prime(\bar{\rho})+ \kappa W * \bar \rho \big) \cdot (\rho u - \bar{\rho} \bar{u}) \, dxd\tau =  0,\nonumber
\end{align}
\begin{align} \label{weakdiff2}
        & \int_\Omega  \bar{u} \cdot (\rho u - \bar{\rho} \bar{u}) \, dx \Big|_{\tau = 0}^{\tau = t} -  \int_0^t \int_\Omega (\partial_\tau \bar{u}) \cdot (\rho u - \bar{\rho} \bar{u}) \, dxd\tau \nonumber\\
       & -   \int_0^t \int_\Omega \nabla \bar{u} : (\rho u \otimes u - \bar{\rho} \bar{u} \otimes \bar{u})\ dxd\tau- \int_0^t \int_\Omega (\nabla \cdot \bar{u})\big(p(\rho)-p(\bar{\rho})\big) \, dxd\tau \\
        = & - \int_0^t \int_\Omega \bar{u} \cdot (\kappa \rho \nabla W * \rho - \kappa \bar{\rho} \nabla W * \bar \rho) \, dxd\tau- \int_0^t \int_{\Omega} \nu \bar u \cdot (\rho u - \bar \rho \bar u) \, dxd\tau.\nonumber
\end{align}
Now, we collect $((\ref{weakenergy}) - (\ref{strongenergy} - (\ref{weakdiff1}) - (\ref{weakdiff2}))$ to obtain
\begin{equation} \label{RE1}
 \begin{split}
 \int_\Omega &  \tfrac{1}{2} \rho |u - \bar{u}|^2 + h(\rho | \bar{\rho}) +\kappa \tfrac{1}{2}(\rho - \bar \rho) \big(W * (\rho - \bar \rho) \big) \, dx\Big|_{\tau = 0}^{\tau = t} \\
 \leq & - \int_0^t \int_\Omega \nu \rho |u|^2 - \nu \bar \rho |\bar u|^2 - \nu \bar u \cdot (\rho u -\bar \rho \bar u) \, dxd\tau \\
 & - \int_0^t \int_\Omega \partial_\tau (- \tfrac{1}{2}|\bar{u}|^2 + h^\prime(\bar{\rho}) + \kappa W * \bar \rho )(\rho - \bar{\rho}) \, dxd\tau \\
& - \int_0^t \int_\Omega \nabla (- \tfrac{1}{2}|\bar{u}|^2+h^\prime(\bar{\rho}) +\kappa W * \bar \rho ) \cdot (\rho u - \bar{\rho} \bar{u}) \, dxd\tau \\
& - \int_0^t \int_\Omega (\partial_\tau \bar{u}) \cdot (\rho u - \bar{\rho} \bar{u})dxd\tau -  \int_0^t \int_\Omega \nabla \bar{u} : (\rho u \otimes u - \bar{\rho} \bar{u} \otimes \bar{u}) \, dxd\tau \\
& - \int_0^t \int_\Omega (\nabla \cdot \bar{u})\big(p(\rho)-p(\bar{\rho})\big) \ dxd\tau +  \int_0^t \int_\Omega \bar{u} \cdot (\kappa\rho \nabla W * \rho - \kappa \bar{\rho} \nabla W * \bar \rho) \, dxd\tau.
  \end{split}
\end{equation} 
The left-hand side of the inequality above was obtained using the symmetry of $W$. 
Recall that the strong solution $(\bar \rho, \bar u)$ satisfies 
$
  \partial_t \bar{u}+\bar{u} \cdot \nabla \bar{u}  = -\nabla \big(h^\prime(\bar{\rho})+\kappa W * \bar \rho \big) - \nu \bar u.
$
Multiplying the above expression by $\rho (u - \bar{u})$ gives:
\begin{align} \label{RE2}
     &\partial_t \big(- \tfrac{1}{2} |\bar{u}|^2\big)(\rho - \bar{\rho})+ \partial_t \bar{u} \cdot (\rho u - \bar{\rho} \bar{u}) +  \nabla \big(- \tfrac{1}{2} |\bar{u}|^2\big) \cdot (\rho u - \bar{\rho} \bar{u}) + \nabla \bar{u} : (\rho u \otimes u-\bar{\rho} \bar{u} \otimes \bar{u}) \nonumber\\
   = &  -\rho\nabla h_1^\prime(\bar{\rho}) \cdot (u-\bar{u})-\rho \nabla \bar{\phi} \cdot (u-\bar{u}) +  \nabla \bar{u} : \rho (u - \bar{u}) \otimes (u - \bar{u}) - \nu \rho \bar u \cdot (u - \bar u).
\end{align}
Substituting (\ref{RE2}) into (\ref{RE1}) results in:
\begin{align} \label{RE3}
 \int_\Omega &   \tfrac{1}{2} \rho |u - \bar{u}|^2 + h(\rho | \bar{\rho}) +\kappa \tfrac{1}{2}(\rho - \bar \rho) \big(W * (\rho - \bar \rho) \big) \, dx\Big|_{\tau = 0}^{\tau = t} \nonumber\\
 \leq 
& - \int_0^t \int_\Omega \nu \rho |u|^2 - \nu \bar \rho |\bar u|^2 - \nu \bar u \cdot (\rho u -\bar \rho \bar u) - \nu \rho \bar u \cdot (u - \bar u) \, dxd\tau \nonumber\\
 & - \int_0^t \int_\Omega \partial_\tau (h^\prime(\bar{\rho}) + \kappa W * \bar \rho )(\rho - \bar{\rho}) \ dxd\tau  - \int_0^t \int_\Omega \nabla (h^\prime(\bar{\rho}) +\kappa W * \bar \rho ) \cdot (\rho u - \bar{\rho} \bar{u}) \, dxd\tau \nonumber\\
& - \int_0^t \int_\Omega (\nabla \cdot \bar{u})\big(p(\rho)-p(\bar{\rho})\big) \ dxd\tau +  \int_0^t \int_\Omega \bar{u} \cdot (\kappa \rho \nabla W * \rho - \kappa \bar{\rho} \nabla W * \bar \rho) \, dxd\tau \\
& + \int_0^t \int_\Omega \rho \nabla h^\prime(\bar \rho) \cdot (u - \bar u) \ dxd\tau +  \int_0^t \int_\Omega \kappa \rho(u - \bar u) \cdot \nabla W * \bar \rho   \, dxd\tau \nonumber\\
& - \int_0^t \int_\Omega \nabla \bar u : \rho (u - \bar u) \otimes (u - \bar u) \, dx d\tau.\nonumber
\end{align} 
Using the continuity equation satisfied by the strong solution, we derive:
\begin{equation} \label{RE4}
\begin{split}
    \partial_t \big(h^\prime(\bar{\rho}) \big)(\rho - \bar{\rho}) & + \nabla h^\prime(\bar{\rho}) \cdot (\rho u -\bar{\rho} \bar{u})  +   (\nabla \cdot \bar{u}) \big(p(\rho) - p(\bar{\rho}) \big) - \nabla h^\prime(\bar{\rho}) \cdot (\rho u -\rho \bar{u}) \\
    & =  (\nabla \cdot \bar{u}) p(\rho | \bar{\rho}).
    \end{split}    
\end{equation} 
Moreover, using the symmetry of $W$ we get:
\begin{equation} \label{RE5}
 \begin{split}
&- \int_\Omega \partial_t (\kappa W * \bar \rho)(\rho - \bar \rho) \, dx - \int_\Omega \nabla (\kappa W * \bar \rho) \cdot (\rho u - \bar \rho \bar u) \, dx \\ 
& + \int_\Omega \bar u \cdot (\kappa \rho \nabla W * \rho - \kappa \bar \rho \nabla W * \bar \rho) \, dx + \int_\Omega \kappa \rho(u -\bar u) \cdot \nabla W * \bar \rho \, dx \\
& =  \int_\Omega \kappa (\rho - \bar \rho) \bar u \cdot \nabla W *(\rho - \bar \rho) \, dx.
 \end{split}
\end{equation}
Replacing (\ref{RE4}) and (\ref{RE5}) in (\ref{RE3}) yields the desired result.
\end{proof}
The next step is to bound the terms $\mathcal{I}_i$ from above using the relative energy $\Psi.$ For $\mathcal{I}_1$ and $\mathcal{I}_2$ we have:
\begin{equation*}
  \begin{split}
    \mathcal{I}_1(t) &=  - \int_0^t \int_{\Omega} \nabla \bar{u}: \rho (u - \bar{u}) \otimes (u - \bar{u}) \, dxd \tau \\
   & \leq C \ \|\nabla \bar{u} \|_{\infty} \int_0^t \int_{\Omega}  \rho |u - \bar{u}|^2 \, dxd\tau \leq C \int_0^t \Psi(\tau) \, d\tau,
  \end{split}
\end{equation*}
\begin{equation*}
  \begin{split}
    \mathcal{I}_2(t) &=  - \int_0^t \int_{\Omega} (\nabla \cdot \bar{u}) p(\rho | \bar{\rho}) \, dxd \tau \\
   & \leq  \|\nabla \cdot \bar{u} \|_{\infty} (\gamma - 1) \int_0^t \int_{\Omega} h(\rho | \bar{\rho}) \, dxd\tau \leq C \int_0^t \Psi(\tau) \, d\tau.
  \end{split}
\end{equation*}

Regarding $\mathcal{I}_3$, we first observe
\begin{align*}
\mathcal{I}_3(t) &= \int_0^t \iint_{\Omega \times \Omega} \kappa \bar u(x) \cdot (\rho - \bar\rho)(x) \nabla W(x-y) (\rho - \bar\rho)(y)\,dxdyd\tau\cr
&=\frac12\int_0^t \iint_{\Omega \times \Omega} \kappa  (\bar u(x) - \bar u(y))\cdot (\rho - \bar\rho)(x) \nabla W(x-y) (\rho - \bar\rho)(y)\,dxdyd\tau
\end{align*}
due to $\nabla W(-x)= -\nabla W(x)$ for $x \in \Omega$. Since 
\[
|\bar u(x) - \bar u(y)||\nabla W(x-y)| \leq C\|\nabla \bar u\|_\infty |x-y|^\alpha,
\]
we proceed as in the proof of Proposition \ref{Cstar} with $\beta = \alpha + d$
obtaining:
\begin{equation*}
\begin{split}
\mathcal{I}_3(t) & \leq C \int_0^t \|(\rho - \bar \rho) I_{\beta} (\rho - \bar \rho) \|_1 \,  d \tau \\
& \leq C \int_0^t \int_\Omega h(\rho | \bar \rho) \, dx d \tau \leq C \int_0^t \Psi(\tau) \, d\tau.
\end{split}
\end{equation*}

In light of the previous bounds, expanding (\ref{relativeinequality1}) further yields
\[ 
\Psi(t) - \Psi(0) + \int_0^t \int_\Omega \nu \rho |u-\bar u|^2 \, dxd\tau \leq C \int_0^t \Psi(\tau) \, d\tau, \quad t \in [0,T[,
\]
from which (\ref{stability1}) follows by Gronwall's lemma. 
\subsection{Proof of Theorem \ref{thmrelax}}\label{section_proofs_RL}
The proof of the next result is analogous to the proof of Proposition \ref{relineprop1}, and is therefore omitted. The only difference to the error term $\bar e$ that is present in this case.
\begin{proposition}
Let $(\rho, u)$ be a dissipative weak solution of (\ref{EW2}), and let $\bar \rho$ be a strong solution of (\ref{GF}). Then, for each $t \in [0,T[,$ the relative energy $\Phi$ satisfies 
\begin{equation*} 
\Phi_\varepsilon(t) - \Phi_\varepsilon(0) + \int_0^t \int_\Omega \rho|u - \bar u|^2 \ dxd\tau \leq \mathcal{J}_1(t) + \mathcal{J}_2(t) + \mathcal{J}_3(t) + \mathcal{J}_4(t),
\end{equation*}
where 
\begin{equation*}
\begin{split}
\mathcal{J}_1(t) & = - \int_0^t \int_\Omega \varepsilon \nabla \bar u : \rho (u - \bar u) \otimes (u - \bar u) \, dxd\tau, \\
\mathcal{J}_2(t) & = - \int_0^t \int_\Omega (\nabla \cdot \bar u) p(\rho | \bar \rho) \, dxd\tau, \\
\mathcal{J}_3(t) & = \int_0^t \int_\Omega \kappa (\rho - \bar \rho) \bar u \cdot \nabla W *(\rho - \bar \rho) \, dxd\tau, \\
\mathcal{J}_4(t) & = - \int_t \int_\Omega \varepsilon \frac{\rho}{\bar \rho} \bar e \cdot (u - \bar u) \, dx d \tau.
\end{split}
\end{equation*}
\end{proposition}
Similarly as above, we obtain the following bounds: 
\[\mathcal{J}_i(t) \leq C \int_0^t \Phi_\varepsilon(\tau) \, d \tau,  \quad i=1,2,3. \]
The term $\mathcal{J}_4$ is treated as follows,
\begin{equation*}
\begin{split}
\mathcal{J}_4(t) & = - \int_0^t \int_\Omega \varepsilon\frac{\rho}{\bar{\rho}} \bar{e} \cdot (u - \bar{u}) \ dxd\tau \leq \int_0^t \int_\Omega \tfrac{1}{2} \rho |u - \bar u|^2 \ dx d\tau +  \int_0^t \int_\Omega \frac{\varepsilon^2}{2} \rho \Big| \frac{\bar e}{\bar \rho} \Big|^2 \ dxd\tau\\ 
& \leq \int_0^t \int_\Omega \tfrac{1}{2} \rho |u - \bar u|^2 \ dx d\tau + C \varepsilon^2 t.
\end{split}
\end{equation*}
Thus,
\[ 
\Phi_\varepsilon(t) - \Phi_\varepsilon(0) + \int_0^t \int_\Omega \tfrac{1}{2} \rho |u - \bar u|^2 \ dx d\tau \leq C \int_0^t \Phi_\varepsilon(\tau) \, d\tau + C\varepsilon^2 t, \quad t \in [0,T[,
\]
from which (\ref{stability2}) follows by Gronwall's lemma.
\section*{Acknowledgments}
NJA acknowledges partial financial support from the Austrian Science Fund (FWF) project  10.55776/F65. JAC was supported by the Advanced Grant Nonlocal-CPD (Nonlocal PDEs for Complex Particle Dynamics: Phase Transitions, Patterns and Synchronization) of the European Research Council Executive Agency (ERC) under the European Union's Horizon 2020 research and innovation programme (grant agreement No. 883363). JAC was also partially supported by the Engineering and Physical Sciences Research Council (EPSRC) under grants EP/T022132/1 and EP/V051121/1. YPC was supported by NRF grant (No. 2022R1A2C1002820). This work originated during a visit of JAC to the King Abdullah University of Science and Technology, where NJA was a PhD student.

\end{document}